\documentclass[a4paper,11pt,french]{article}
\usepackage[OT2,T1]{fontenc}
\usepackage[latin1]{inputenc}
\makeatletter

 \usepackage{amsmath,amsthm}
 \usepackage{amsfonts}
\usepackage[all]{xy}
\DeclareSymbolFont{cyrletters}{OT2}{wncyr}{m}{n}
\DeclareMathSymbol{\Sha}{\mathalpha}{cyrletters}{"58}
\theoremstyle{plain}    
 \newtheorem{thm}{Th\'eor\`eme}[section]
 \newtheorem{defn}[thm]{D\'efinition}
 \numberwithin{equation}{section} 
 \numberwithin{figure}{section} 
 \theoremstyle{plain}
 \theoremstyle{plain}
 \newtheorem{prop}[thm]{Proposition}
\newtheorem{corollaire}[thm]{Corollaire}
 \newtheorem{lemme}[thm]{Lemme}

{\setlength{\parindent}{0mm}
{\setlength{\parskip}{3mm}

 \def\Gal{\operatorname{Gal}}
 
 \def\Cap{\operatorname{Cap}}
 
 \def\mod{\operatorname{mod}}
 \def\res{\operatorname{res}}

 \def\coker{\operatorname{coker}}
 \def\Hom{\operatorname{Hom}}
 \def\supp{\operatorname{supp}}
  \def\rg{\operatorname{rg}}
   \def\Q{\mathbb Q}
\def\Z{\mathbb Z}       \def\O{\mathcal O} \def\Q{\mathbb Q} \def\F{\mathbb F} \def\R{\mathbb R}    \def\L{\mathcal L}
\def\G{\mathcal G}   
\def\WK{ WK^{\acute{e}t}}
\def\K{ K^{\acute{e}t}}

\def\scd{scd}
\def\Ver{Ver} \def\res{res} \def\cor{cor} 

\makeatother

\usepackage{babel}
\makeatother

\begin{document}
\title{Sur la pro-$p$-extension localement cyclotomique maximale d'un corps de nombres}
\author{Romain Validire}
\maketitle
\textbf{R\'esum\'e.} Soit $p$ un nombre premier et $F_{\infty}$ la $\Z_p$-extension cyclotomique d'un corps de nombres $F$. Nous \'etudions le groupe de Galois $\mathcal G_{F_{\infty}}'$ sur $F_{\infty}$ de la pro-$p$-extension non ramifi\'ee, $p$-d\'ecompos\'ee maximale de $F_{\infty}$. La question de la pro-$p$-libert\'e de $\mathcal G_{F_{\infty}}'$ a d\'ej\`a \'et\'e \'evoqu\'e par de nombreux auteurs. Dans cet article, nous caract\'erisons la pro-$p$-libert\'e de $\mathcal G_{F_{\infty}}'$ en termes de descente galoisienne pour certains noyaux de localisation en cohomologie galoisienne : \textit{les noyaux sauvages \'etales}. Nous en d\'eduisons des crit\`eres effectifs pour la non pro-$p$-libert\'e de ce groupe. 

\textbf{Abstract.} Let $p$ be a prime number and $F_{\infty}$ be the cyclotomic $\Z_p$-extension of a number field $F$. We consider the Galois group $\mathcal G_{F_{\infty}}'$ over $F_{\infty}$ of the maximal unramified, $p$-decomposed, pro-$p$-extension of $F_{\infty}$. The question whether $\mathcal G_{F_{\infty}}'$ is free pro-$p$ was already asked by many authors. In this article, we highlight a link between the freeness of $\mathcal G_{F_{\infty}}'$ and the Galois descent for some localisation kernels : \textit{the \'etale wild kernels}. Then we give explicit criterions to show that $\mathcal G_{F_{\infty}}'$ is not a free pro-$p$-group.

\section*{Introduction}
Soient $p$ un nombre premier et $F$ une extension alg\'ebrique du corps des rationnels $\Q$. On d\'esigne par $\L_F$ la pro-$p$-extension non ramifi\'ee maximale de $F$ et on note $\G_F=\Gal(\L_F/F)$ son groupe de Galois sur $F$.\\
M\^eme dans le cas o\`u $F$ est un corps de nombres (i.e. $F/\Q$ est fini), la structure du pro-$p$-groupe $\G_F$ n'est pas bien connue. La th\'eorie du corps de classes nous montre que $\G_F^{ab}$ est isomorphe \`a la $p$-partie du groupe des classes de $F$ (il s'agit donc d'un groupe fini). En 1964, Golod et Shafarevitch ont donn\'e les premiers exemples de corps $F$ pour lesquels $\G_F$ est infini. De nombreuses questions concernant la structure de ces groupes restent cependant ouvertes (dimension cohomologique, analyticité $p$-adique,...). \\
Soit $F_{\infty}$ une $\Z_p$-extension de $F$. Une strat\'egie classique pour \'etudier le groupe $\G_{F}$, est d'\'etudier le groupe $\G_{F_{\infty}}$ via la théorie d'Iwasawa des $\Z_p$-extensions (c'est par exemple le point de vue adopt\'e dans \cite{O}); l'id\'ee est ensuite de \textit{descendre} les r\'esultats au niveau de $F$. Dans cet article, nous nous int\'eressons aux objets suivants : 
\begin{itemize}
\item $F_{\infty}$ est la $\Z_p$-extension \textit{cyclotomique} d'un corps de nombres $F$.
\item $\L_{F_{\infty}}'$ est la pro-$p$-extension non ramifi\'ee, \textit{$p$-decompos\'ee}, maximale de $F_{\infty}$ et le groupe de Galois $$\G_{F_{\infty}}'=\Gal(\L_{F_{\infty}}'/F_{\infty}).$$
\end{itemize}
Le groupe $\G_{F_{\infty}}'$ est un quotient naturel de $\G_{F_{\infty}}$ et leurs structures sont fortement reliées. Dans cet article, nous nous intéressons à la question suivante :
$$\mbox{Le groupe } \G_{F_{\infty}}'\mbox{ est-il un pro-}p\mbox{-groupe libre ?}$$
On d\'esigne par $A_F'$ la $p$-partie du groupe des $p$-classes de $F$ et par $\mu_n$ le groupe des racines $n$-i\`emes de l'unit\'e. Faisons un bref historique de la question précédente. Supposons que le corps $F$ est \`a multiplication complexe et que la \textit{partie plus} de $A_F'$ est triviale (par exemple, $F$ est un corps cyclotomique $\Q(\mu_p)$ satisfaisant la \textit{conjecture} de Vandiver) ; sous ces hypothèses, on pensait que le groupe $\G_{F_{\infty}}'$ était pro-$p$-libre  (cf. \cite[Theorem 1]{Wi} et \cite[Exemples 3.2]{N2}). Cependant, en 2003, des r\'esultats de W. McCallum et R. Sharifi (cf. \cite{MS}) mettent en porte-\`a-faux les r\'esultats ant\'erieurs. En particulier, on a l'exemple suivant (cf. \cite{Sha}): pour $p=157$ et $F=\Q(\mu_{157})$ le groupe $\G_{F_{\infty}}'$ est isomorphe au groupe ab\'elien $\Z_p^2$. Aujourd'hui, les seuls exemples connus de couples $(F,p)$ pour lesquels $\G_{F_{\infty}}'$ est pro-$p$-libre sont les couples $(F,p)$ pour lesquels $\G_{F_{\infty}}'$ est trivial ou isomorphe \`a $\Z_p$.

Dans cet article, notre but est de caract\'eriser la pro-$p$-libert\'e de $\G_{F_{\infty}}'$ en termes de descente galoisienne pour les \textit{noyaux sauvages \'etales}. Pour tout $i\geq 1$, les noyaux sauvages \'etales $\WK_{2i}(F)$ attach\'e \`a $F$ et $p$ peuvent \^etre vus comme des versions \textit{tordues} du $p$-groupe des classes $A_F'$. Ils sont d\'efinis comme des noyaux de localisation en cohomologie galoisienne. Pour $i=1$, le groupe $\WK_{2}(F)$ est isomorphe \`a la $p$-partie du noyau sauvage \textit{classique} i.e. le noyau des symboles de Hilbert dans le groupe $K_2(F)$.\\

Dans une premi\`ere partie, nous \'etablissons un r\'esultat (Proposition \ref{kawada2}) qui caract\'erise la libert\'e d'un pro-$p$-groupe au moyen du morphisme de transfert. Dans une seconde partie, nous rappelons une description classique, due \`a Schneider, des noyaux sauvages en termes d'un certain module d'Iwasawa. L'objet de la troisi\`eme partie est d\'etablir le r\'esultat principal (th\'eor\`eme \ref{thm-desc-loc-cycl}) qui relie le comportement galoisien des noyaux sauvages \`a la structure de $\G_{F_{\infty}}'$.\\
Notons que l'\'etude de $\G_{F_{\infty}}'$ (pr\'ecisement la finitude) via les noyaux sauvages a d\'ej\`a \'et\'e entreprise dans \cite{As}. L'auteur retrouve le crit\`ere d'infinitude pour $\G_{F_{\infty}}'$ donn\'e dans \cite{JS}. Celui-ci provient d'une in\'egalit\'e \`a la Golod et Shafarevitch. Dans la derni\`ere partie, nous obtenons \`a l'aide du th\'eor\`eme \ref{thm-desc-loc-cycl} un crit\`ere de \textit{non} pro-$p$-libert\'e (th\'eor\`eme \ref{thm-ineg}) pour $\G_{F_{\infty}}'$ ; celui-ci s'exprime au moyen d'une in\'egalit\'e similaire \`a celle de \cite{JS}. Le crit\`ere nous permet de construire effectivement des couples $(F,p)$ pour lesquels $\G_{F_{\infty}}'$ n'est pas libre. Enfin, toujours comme cons\'equence du th\'eor\`eme \ref{thm-desc-loc-cycl}, nous obtenons le corollaire \ref{cor-partieplus} qui montre que pour un corps $F$ à multiplication complexe, la trvialit\'e de la partie plus de $A_F'$ est une condition \textit{n\'ecessaire} pour la libert\'e de $\G_{F_{\infty}}'$ (nous retrouvons ainsi le r\'esultat \cite[Proposition 3.3]{Wi}). 

\section{Pr\'eliminaires cohomologiques}
\subsection{Pro-$p$-libert\'e et transfert}
Fixons un nombre premier $p$. Dans ce paragraphe nous rappelons des r\'esultats standards sur la cohomologie des groupes et la notion de pro-$p$-libert\'e. On peut trouver la plupart des r\'esultats \'enonc\'es dans \cite{NSW} ou \cite{Se}.
 
Etant donn\'e un groupe ab\'elien $M$ localement compact, on rappelle que $M^*$ d\'esigne le dual de Pontryagin de $M$. C'est le groupe $\Hom(M,\R/\Z)$ des homomorphismes \textit{continus} de $M$ vers $\R/\Z$. On a une dualit\'e parfaite $M \simeq (M^*)^*$ qui transforme les groupes discrets en groupes compacts.
 
Lorsque $M$ est un pro-$p$-groupe ou un groupe discret de $p$-torsion $$M^*=\Hom(M,\Q_p/\Z_p).$$

Soit $G$ un groupe profini et $M$ un $G$-module \textit{compact}. Pour tout entier $n\geq 0$, les groupes d'homologie sont d\'efinis par dualit\'e \`a partir des groupes de cohomologie :
$$H_n(G,M):=\left(H^n(G,M^*)\right)^*.$$

Etant donn\'e un pro-$p$-groupe $G$, on note $G^{ab}$ l'abelianis\'e de $G$. C'est le quotient de $G$ par l'adh\'erence de son sous-groupe d\'eriv\'e $[G,G]$.\\
En outre, on a $$G^{ab}\simeq H_1(G,\Z_p).$$
La notation $H\triangleleft G$ signifie que $H$ est un sous-groupe distingu\'e de $G$. Enfin, on note $d(G):=\dim_{\F_p}(H^1(G,\F_p))$ le $p$-rang de $G$. C'est le nombre minimal de g\'en\'erateurs de $G$.

Etant donn\'e un pro-$p$-groupe $G$, on note $cd(G)$ (resp. $scd(G)$) la dimension cohomologique (resp. dimension cohomologique stricte) du pro-$p$-groupe $G$. Pro-$p$-libert\'e et dimension cohomologique sont reli\'ees par la proposition suivante : 
\begin{prop}
Soit $G$ un pro-$p$-groupe non trivial. On a l'\'equivalence :\\
(i) $G$ est pro-$p$-libre.

(ii) $cd(G)=1$.
\end{prop}

Il en d\'ecoule la caract\'erisation suivante :
\begin{prop}\label{scd}
Le groupe $G$ est pro-$p$-libre si et seulement si $\scd(G)=2$ et $G^{ab}$ est sans $p$-torsion.
\end{prop}

Etant donn\'es un pro-$p$-groupe $G$ et un sous-groupe $H$ \textit{d'indice fini} de $G$, on d\'efinit (cf. \cite[VII, \S 8]{Se-CL}) un morphisme canonique appel\'e \textit{transfert} :
$$\Ver : G^{ab}\rightarrow H^{ab},$$
qui s'identifie \`a l'homomorphisme de restriction entre groupes d'homologie :
$$\res : H_1(G,\Z_p)\rightarrow H_1(H,\Z_p).$$
Par dualit\'e, le transfert s'identifie \`a l'application de corestriction entre les groupes de cohomologie :
$$\cor : H^1(H,\Q_p/ \Z_p)\rightarrow H^1(G,\Q_p /\Z_p) .$$

La proposition suivante fait un lien entre l'\'etude du transfert et la dimension cohomologique stricte (cf. \cite[Theorem 3.6.4]{NSW}). Elle caract\'erise les groupes profinis dont la dimension cohomologique stricte est \'egale \`a $2$. Ces groupes sont particuli\`erement importants, notamment parce qu'ils apparaissent naturellement en th\'eorie du corps de classes. Cette proposition est une combinaison de r\'esultats d\^us \`a J.-P. Serre et J. Tate.
\begin{prop}\label{kawada}
Soit $G$ un groupe profini non nul. Les propositions suivantes sont \'equivalentes :\\
(i) La dimension cohomologique stricte de $G$ est $2$.

(ii) Pour tout couple de sous-groupes ouverts distingu\'es $V \triangleleft U$, le transfert induit un isomorphisme $U^{ab}\stackrel{\simeq}{\rightarrow}(V^{ab})^{U/V}$.
\end{prop}

Nous proposons une adaptation de la proposition pr\'ec\'edente ; celle-ci permet de caract\'eriser exactement la pro-$p$-libert\'e d'un groupe au moyen du tansfert.

\begin{prop}\label{kawada2}
Soit $G$ un pro-$p$-groupe non nul. Les propositions suivantes sont \'equivalentes :\\
(i) $G$ est pro-$p$-libre.

(ii) Pour tout sous-groupe ouvert $U\triangleleft G$ le module $U^{ab}$ est $\Z_p$-libre et pour tout couple de sous-groupes ouverts $V \triangleleft U$, le transfert induit une injection
$$\mathcal V : U^{ab}\otimes \Q_p/\Z_p \hookrightarrow V^{ab}\otimes \Q_p/\Z_p ,$$
o\`u $\mathcal V$ est d\'efinie par $\mathcal V(x\otimes \frac{1}{p^n})=\Ver(x)\otimes\frac{1}{p^n}$.

\end{prop}

\begin{proof}
la proposition est une cons\'equence des propositions \ref{scd} et \ref{kawada}.
\end{proof}
\subsection{$K$-th\'eorie des anneaux d'entiers et cohomologie galoisienne}
Commen\c cons par donner quelques notations.\\
Soient $p$ un nombre premier et $F$ un corps de nombres, on note :
\begin{quote}
\item $\mu_{p^n}$ le groupe des racines $p^n$-i\`emes de l'unit\'e,
\item $\mu_{p^{\infty}}:=\bigcup_{n\geq 1}\mu_{p^n}$ et $\Z_p(1):=\varprojlim \mu_{p^n}$,
\item $F_v$ le localis\'e de $F$ en $v$, o\`u $v$ d\'esigne une place de $F$,
\item $S_p=S_p(F)$ l'ensemble des places de $F$ divisant $p$ et $\infty$,
\item $S=S_F$ un ensemble fini de places de $F$ contenant $S_p$,
\item $\O_F^S$ l'anneau des $S$-entiers de $F$,
\item $A_F'$ la $p$-partie du groupe du groupe de classes de $\O_F^{S_p}$,
\item $F^S$ l'extension $S$-ramifi\'ee maximale de $F$ et $G_F^S:=\Gal(F^S/F)$,
\item $r_1(F)$ (resp. $r_2(F)$) le  nombre de places r\'eelles (resp. de places complexes \`a conjuguaison pr\`es) de $F$.
\end{quote}

Enfin pour tout $i\geq 0$, on pose
$$A(i):=A\otimes \Z_p(i)=A\otimes \underbrace{\Z_p(1)\otimes \cdots \otimes \Z_p(1),}_{\mbox{$i$ fois}}$$
le $i$-i\`eme tordu \`a la Tate de $A$.

\textbf{Hypoth\`ese :} dans toute la suite du texte, lorsque $p=2$ on suppose que $\sqrt{-1} \in F$.

Pour $i\geq 1$, et $k=1,2$ les groupes de $K$-th\'eorie \'etale $\K_{2i+2-k}(\O_F^S)$ introduits par Dwyer et Friedlander (cf. \cite{KM}) sont isomorphe aux groupes de cohomologie galoisienne continue :
$$H_{cont}^k(G_F^S,\Z_p(i+1)):=\varprojlim H^k(G_F^S,\Z/p^n(i+1)).$$

Les r\'esultats de Quillen et Borel sur la $K$-th\'eorie alg\'ebrique ainsi que les r\'esultats de Soul\'e sur la surjectivit\'e des caract\`eres de Chern (cf \cite{Sou}) nous donnent les propri\'et\'es suivantes :
\begin{itemize}
\item Pour tout $i \geq 1$, les groupes $\K_{2i}(\O_F^S)$ sont \textit{finis}.

\item Pour tout $i\geq 1$, les groupes $\K_{2i+1}(\O_F^S)\simeq \K_{2i+1}(F)$ sont de \textit{type fini} sur $\Z_p$. Pr\'ecis\'ement, on a :

$$rg_{\Z_p}(\K_{2i+1}(F)) =
\left\{
\begin{array}{cc}
r_1+r_2  & \mbox{ si } i \mbox{ est pair, } \\
r_2      & \mbox{ si } i \mbox{ est impair }
\end{array}
\right.
$$
\end{itemize}

Rappelons qu'une extension de corps de nombres $L/F$ est une $p$-extension lorsque $L/F$ est une extension galoisienne finie et $\Gal(L/F)$ est un $p$-groupe.

On fixe un ensemble fini $S$ de places de $F$ contenant les places $p$-adiques et les places \`a l'infini. Par abus, on note toujours $S:=S_L$ l'ensemble des places de $L$ au-dessus des places contenues dans $S$. On s'int\'eresse maintenant au comportement galoisien des $K$-groupes dans les $p$-extensions $S$-ramifi\'ees.

Soit $M$ un groupe ab\'elien et $G$ un groupe op\'erant sur $M$. On note 
\begin{itemize}
\item $M^G$ le sous-groupe de $M$ des \'el\'ements invariants par $G$.  
\item $M_G$ le quotient $M/I_GM$, o\`u $I_G$ est le sous-groupe de $\Z[G]$ engendr\'e par les $g-1$, $g \in G$.
\end{itemize}

Soit $L/F$ une $p$-extension de groupe de Galois $G$, non ramifi\'ee hors de $S$. Pour $i\geq 1$ et $k\in\{1,2\}$ on a un morphisme d'extension :
$$e_{k,i} : \K_{2i+2-k}(\O_F^S) \rightarrow \K_{2i+2-k}(\O_L^S)^G,$$
qui s'interpr`ete cohomologiquement comme un morphisme de restriction.

Les noyaux des morphismes d'extension sont appel\'es \textit{noyaux de capitulation}. Ils jouent un r\^ole central dans la suite de l'article.
\begin{defn}
Pour toute extension $S$-ramifi\'ee finie $L/F$ et tout entier $i\geq 1$, on note :
$$\Cap_i^{S}(L/F):=\ker(\K_{2i}(\O_{F}^S)\rightarrow \K_{2i}(\O_{L}^S)).$$
Si $\L/F$ est une extension alg\'ebrique, on pose $\Cap_i^S(\L/F)=\varinjlim \Cap(L/F)$, o\`u $L$ parcourt les sous-extensions finies de $\L/F$.
\end{defn}

Tout comme les noyaux de capitulation des groupes de classes qui s'expriment en termes de cohomologie des \textit{unit\'es}, les noyaux $\Cap_i^{S}(L/F)$ sont reli\'es \`a la cohomologie des $K$-groupes \textit{impairs}. Pour $i=1$, le r\'esultat suivant est \`a rapprocher d'un r\'esultat de B. Kahn (cf \cite{Ka}). Pour la g\'en\'eralisation \`a tout $i\geq 1$, on renvoie \`a \cite[Theorem 1.2]{KM}.\\
Les groupes $\hat{H}^*(G,.)$ d\'esignent les groupes de cohomologie modifi\'es (cf. par exemple \cite[Chapitre VIII]{Se-CL})
\begin{thm}
Soit $L/F$ une $p$-extension, non ramifi\'ee en dehors de $S$ et $G:=\Gal(L/F)$. Alors
$$\Cap_i^{S}(L/F)\simeq \hat{H}^{-1}(G,\K_{2i}(\O_L^S)) \simeq H^1(G,\K_{2i+1}(L)) , \mbox{\ et}$$
$$\coker(e_{2,i})\simeq \hat{H}^{0}(G,\K_{2i}(\O_L^S)) \simeq H^2(G,\K_{2i+1}(L)).$$
De plus, lorsque $L/F$ est cyclique, le quotient de Herbrand
$$h(G,\K_{2i+1}(L))=\frac{|H^2(G,\K_{2i+1}(L))|}{|H^1(G,\K_{2i+1}(L))|}$$
est trivial.
\end{thm}

\textbf{Remarques.}
\begin{itemize}
\item Les noyaux $\Cap_i^{S}(L/F)$ ne d\'ependent pas de l'ensemble $S$ contenant $S_p\cup S_{\infty}$ et les premiers ramifi\'es dans $L/F$. On note d\'esormais $$\Cap_i(L/F):=\Cap_i^{S}(L/F).$$

\item Trouver une minoration int\'eressante de $|\Cap_i(L/F)|$ est en g\'en\'eral plus difficile (le probl\`eme est soulev\'e par B. Kahn dans l'introduction de \cite{Ka}). Dans \cite{AM}, les auteurs donnent une minoration de cet ordre sous certaines hypoth\`eses de ramification pour $L/F$.
\end{itemize}

La proposition suivante est classique et porte sur la trivialit\'e des noyaux de capitulation dans une $p$-extension.
\begin{prop}\label{cap-nul} Soient $L/F$ et $L'/F$ des $p$-extensions finies, $S$-ramifi\'ees avec $L\subset L'$. Alors 
$$\Cap_i(L'/F)=0\  \mbox{si et seulement si }\  \Cap_i(L'/L)=0 \mbox{ et } \Cap_i(L/F)=0$$
\end{prop}

\subsection{L'alg\`ebre d'Iwasawa}
Nous terminons cette partie par quelques rappels sur l'alg\`ebre d'Iwasawa. Soit $\Lambda := \Z_p[[T]]$, l'alg\`ebre des s\'eries formelles en $T$ \`a coefficient dans $\Z_p$. Soient $\Gamma$ un pro-$p$-groupe multiplicatif isomorphe \`a $\Z_p$ et $\gamma$ un g\'en\'erateur topologique de $\Gamma$. Pour tout entier $n\geq 0$, on pose $\Gamma_n := \Gamma^{p^n}$. L'ag\`ebre $\Lambda$ est topologiquement isomorphe \`a l'alg\`ebre de groupes compl\`ete $\Z_p[[\Gamma]]=\varprojlim \Z_p[[\Gamma/\Gamma_n]]$ via l'application $\gamma \rightarrow 1+T$.\\
Si $M$ est un $\Lambda$-module de type fini, il existe des polyn\^omes distingu\'es irr\'eductibles $f_i\in \Z_p[T]$, des entiers $n_i$, $m_j$ et $r_M$ tels que $M$ est pseudo-isomorphe \`a
$$\Lambda^r\oplus \bigoplus_{i=1}^{n}\Lambda/(f_i)^{n_i} \oplus \bigoplus_{j=1}^{m}\Lambda/(p)^{m_j},$$
Les entiers $r_M$, $\lambda_M:=\sum_{i=1}^{n}n_i\deg(f_i)$ et $\mu_M:=\sum_{j=1}^{m}m_i$ sont les \textit{invariants d'Iwasawa} de $M$. Le polyn\^ome $f_M(T)=p^{\mu_X}\Pi_{i=1}^{n}f_i(T)$ est appel\'e polyn\^ome caract\'eristique de $M$.

\bigskip

Rappelons que $F_{\infty}=\bigcup_{n\geq 0}F_n$ d\'esigne la $\Z_p$-extension cyclotomique de $F$ et $\Gamma:=\Gal(F_{\infty}/F)$. On note $\G_{F_{\infty}}'$ le groupe de Galois sur $F_{\infty}$ de la pro-$p$-extension non ramifi\'ee, $p$-d\'ecompos\'ee, maximale de $F_{\infty}$ (c'est aussi la pro-$p$-extension non ramifi\'ee, d\'ecompos\'ee en \textit{toute} place, maximale de $F_{\infty}$).

Dans cet article, nous consid\'erons le $\Lambda$-module $$X_{F_{\infty}}':=H_1(\G_{F_{\infty}}',\Z_p)\simeq \G_{F_{\infty}}'^{ab}.$$
C'est le groupe de Galois sur $F_{\infty}$ de la pro-$p$-extension non ramifi\'ee, $p$-d\'ecompos\'ee, \textit{ab\'elienne}, maximale de $F_{\infty}$. Par la th\'eorie du corps de classe c'est aussi :
$$X_{F_{\infty}}'\simeq \varprojlim A_{F_n}',$$ o\`u la limite projective est prise sur les morphismes de normes.\\
Le $\Lambda$-module $X_{F_{\infty}}'$ est de type fini et de torsion (i.e. $r_{X_{F_{\infty}}'}=0$).\\
Il est conjectur\'e que $\mu_{X_{F_{\infty}}'}=0$. Ce r\'esultat est vrai lorsque $F/\Q$ est ab\'elien, d'apr\`es un th\'eor\`eme de Ferrero et Washigton (cf. \cite{FW}).\\
\textbf{Remarque.} L'hypoth\`ese $\mu_{X_{F_{\infty}}'}=0$ \'equivaut \`a dire que $\G_{\infty}'$ est un pro-$p$-groupe de type fini (i.e. $d(\G_{\infty}')< +\infty$).

\medskip

Rappelons pour finir les notions de \textit{co-adjoint} et \textit{suites admissibles} qui seront utiles dans la suite. Les $\Lambda$-modules consid\'er\'es seront de type fini.\\
Consid\'erons l'application naturelle de localisation :
\[\displaystyle{\Psi_M : M\rightarrow \bigoplus_{\mathfrak p\in \supp(M)}M_{\mathfrak p}},\]
o\`u $\supp(M)$ d\'esigne l'ensemble des id\'eaux premiers de hauteur $1$ de $\Lambda$ disjoints de l'id\'eal $(f_M(T))$ et $M_{\mathfrak p}$ le localis\'e de $M$ en $\mathfrak p$.
On d\'efinit alors :
\begin{defn}(\textbf{et Proposition.})

$\coker{\Psi_M}:=\beta(M)$ est le \textit{co-adjoint} de $M$.

$\ker{\Psi_M}:=M^0$ est le sous-$\Lambda$-module fini maximal de $M$.
\end{defn}

La notion de suite \textit{admissible} nous permet d'obtenir des repr\'esentation plus explicites du co-adjoint et du sous-module fini maximale de $M$.

\begin{defn}
Une suite $\{\pi_n\}_{n\geq 0}$ d'\'el\'ements non nuls de $\Lambda$ est \textit{$M$-admissible} si
\begin{enumerate}
\item $\pi_{0}\in (p,T)$ et pour tout $n\geq 1$ on a $\pi_{n+1}\in \pi_{n}(p,T)$.
\item Les diviseurs $\pi_n$ et $car(M)$ sont \'etrangers (i.e. $M/\pi_{n}$ est fini).
\end{enumerate}
\end{defn}

\begin{thm}\label{thm-suiteadmissible}
Soit $M$ un $\Lambda$-module de torsion et de type fini et $\{\pi_n\}_{n\geq 0}$ une suite $M$-admissible. Alors on a un isomorphisme de $\Lambda$-modules
$$\beta(M)\simeq \varinjlim M/\pi_nM.$$
\end{thm}

Enfin le sous-module fini maximal $M^0$ de $M$ se d\'ecrit de la mani\`ere suivante :
\begin{thm}Soient $M$ un $\Lambda$-module de type fini de torsion et $\{\pi_n\}$ une suite $M$-admissible.
Pour tout $m\geq n\geq 0,$
$$\ker\left(M/\pi_n  \stackrel{\frac{\pi_{m}}{\pi_n}}{\rightarrow}  M/\pi_m\right)\subseteq M^0/\pi_n.$$
De plus, pour tout $n\gg 0$ et tout $m\geq n$ suffisamment grand :
$$M^0\simeq\ker\left(M/\pi_n  \stackrel{\frac{\pi_{m}}{\pi_n}}{\rightarrow}  M/\pi_m\right).$$
\end{thm}
Il en r\'esulte imm\'ediatement le corollaire suivant :
\begin{corollaire}\label{cor-noyaudecap} Pour tout $n\gg 0$, la projection $M\rightarrow M/\pi_n$ induit un isomorphisme :
$$M^0 \simeq \ker\left(M/\pi_n \rightarrow \beta(M)\right).$$
\end{corollaire}

\section{Les noyaux sauvages \'etales}
\subsection{Noyaux de localisation}
Pour tout $i\in \Z$, P. Schneider a introduit (cf. \cite{S})  les noyaux de localisation
$$\Sha_S^2(F,\Z_p(i+1)):=\ker(H^2(G_F^S,\Z_p(i+1))\rightarrow \oplus_{v\in S}H^2(F_v,\Z_p(i+1))).$$

Le groupe $\Sha_{S_p}^2(F,\Z_p(1))$ s'identifie canoniquement au groupe $A_F'$. Pour tout $i\geq 1$, le groupe $\Sha_S^2(F,\Z_p(i+1))$ s'identifie canoniquement \`a un sous-groupe de $\K_{2i}(\O_F^S)$. Ainsi $\Sha_S^2(F,\Z_p(i+1))$ est un $p$-groupe ab\'elien fini. Il ne d\'epend pas de l'ensemble $S$ (contenant $S_p\cup S_{\infty}$). On adopte alors la notation suivante :
$$\WK_{2i}(F):=\Sha_S^2(F,\Z_p(i+1)).$$
Le groupe $\WK_{2i}(F)$ est appel\'e \textit{$2i$-i\`eme noyau sauvage \'etale}. Cette appellation est due \`a T. Nguyen Quang Do (cf. \cite{N1}) ; elle est justifi\'ee par le fait que pour $i=1$, le groupe $\Sha_S^2(F,\Z_p(2))$ s'identifie, d'apr\`es les r\'esultats de Tate, \`a la $p$-partie du noyau sauvage usuel $WK_{2}(F)$. Noyaux sauvages et $K$-groupes pairs sont reli\'es par les suites exactes :
\begin{eqnarray*}\label{se-WK-K}
0\rightarrow \WK_{2i}(F) \rightarrow \K_{2i}(\O_F^S) \rightarrow \oplus_{v\in S}H^2(F_v,\Z_p(i+1))\rightarrow H^0(F,\Q_p/\Z_p(-i))^*\rightarrow 0.
\end{eqnarray*}

Noyaux sauvages et $p$-groupes de classes sont reli\'es par la proposition suivante :
\begin{prop}\label{prop-surj}
Fixons un entier $i\geq 1$. On suppose que $F$ contient $\mu_{2p}$ et qu'au moins un premier $p$-adique de $F$ est totalement ramifi\'ee dans $F_{\infty}/F$. Alors il existe une application surjective :
$$\WK_{2i}(F)/p\rightarrow  A_F'/p(i).$$
\end{prop}

\subsection{Th\'eorie d'Iwasawa des noyaux sauvages}
On fait maintenant le lien entre les noyaux sauvages d\'efinis pr\'ec\'edemment et la th\'eorie d'Iwasawa.

Rappelons que $\Gamma=\Gal(F_{\infty}/F)$ o\`u $F_{\infty}/F$ est la $\Z_p$-extension cyclotomique et que pour tout $n\geq 0$, on note $\Gamma_n:=\Gal(F_{\infty}/F_n)=\Gamma^{p^n}$. On pose $E=F(\mu_p)$ et $\Delta = \Gal(E/F)$. Enfin on note $d$ l'ordre de $\Delta$. Pour simplifier, on pose $X_{\infty}':=X_{F_{\infty}}'$.

Par dualit\'e de Poitou-Tate et mont\'ee dans la $\Z_p$-extension cyclotomique P. Schneider donne une description du $2i$-i\`eme noyau sauvage de $E$ comme co-descendu d'un $\Lambda$-module (cf \cite{S}, \S 6 lemma 1). En consid\'erant les co-invariants sous l'action de $\Delta$ on obtient le th\'eor\`eme suivant :

\begin{thm}\label{schneider}
Pour tout entier positif $i$ non nul tel que $i \equiv 0 \mod d$, il existe un isomorphisme canonique :
$$ \WK_{2i}(F) \simeq \left(X_{\infty}'(i)\right)_{\Gamma} $$
\end{thm}

Posons $\WK_{2i}(F_{\infty}):=\varinjlim \WK_{2i}(F_n)$, o\`u la limite est prise sur les morphismes d'extensions. 

On tire du th\'eor\`eme \ref{schneider} la description \`a l'infini suivante :
\begin{prop}\label{isom-description}
Si $\mu_{X_{\infty}'} = 0$ alors pour tout entier positif $i$ non nul tel que $i \equiv 0 \mod d$, on a un isomorphisme canonique :
$$\WK_{2i}(F_{\infty}) \simeq X_{\infty}' \otimes \Q_p/\Z_p(i).$$
\end{prop}
\begin{proof}
L'hypoth\`ese $\mu_{X_{\infty}'} = 0$ implique que la suite $\{p^n\}$ est $X_{\infty}'(i)$-admissible. Donc d'apr\`es le th\'eor\`eme \ref{thm-suiteadmissible}, on a l'isomorhisme :
\begin{eqnarray*}
\beta(X_{\infty}'(i)) & \simeq & \varinjlim (X_{\infty}'(i))/p^n\\
                      & =      & X_{\infty}'\otimes \Q_p/\Z_p(i).
\end{eqnarray*}
D'autre part, la finitude de $\WK_{2i}(F_n) \simeq \left(X_{\infty}'(i)\right)_{\Gamma_n}$, montre que $\{(1+T)^{p^n}-1\}$ est une suite $X_{\infty}'(i)$-admissible. Donc
\begin{eqnarray*}
\beta(X_{\infty}'(i)) & \simeq & \varinjlim (X_{\infty}'(i))_{\Gamma_n}\\
                      & :=     & \WK_{2i}(F_{\infty}).
\end{eqnarray*} 
\end{proof}

Rappelons enfin la description classique des noyaux de capitulation dans la $\Z_p$-extension cyclotomique. Pour tout $n\geq 0$, et tout entier positif non nul $i \equiv 0 \mod d$, on a
\begin{eqnarray*}
\Cap_{i}(F_n/F) &    =   &\ker(\WK_{2i}(F)\rightarrow \WK_{2i}(F_n))       \\
                  & \simeq &\ker((X_{\infty}'(i))_{\Gamma}\rightarrow (X_{\infty}'(i))_{\Gamma_n})\\
                  & \subseteq & \left((X_{\infty}')^0(i)\right)_{\Gamma}
\end{eqnarray*}
On pose $\Cap_i(F_n) = \ker(\WK_{2i}(F_n)\rightarrow \WK_{2i}(F_{\infty}))$. Alors pour tout $n$ suffisamment grand, le corollaire \ref{cor-noyaudecap} nous donne :
\begin{eqnarray*}
\Cap_{i}(F_n) & \simeq & (X_{\infty}')^0(i)
\end{eqnarray*}

\begin{prop}\label{desc-cycl}Les propositions suivantes sont equivalentes :

(i) Le $\Z_p$-module $X_{\infty}'=(\G_{\infty}')^{ab}$ est ab\'elien libre.

(ii) L'invariant $\mu_{X_{\infty}'}$ est nul et pour tout entier positif non nul $i$ tel que $i\equiv 0 \mod d$, on a la descente galoisienne
$$\WK_{2i}(F)\simeq \WK_{2i}(F_{\infty})^{\Gamma}.$$
\end{prop}
\begin{proof}
La fl\^eche $\WK_{2i}(F)\rightarrow \WK_{2i}(F_{\infty})^{\Gamma}$ est toujours surjective (cf. \cite[Lemma 1.1]{LMN}).\\
Regardons l'injectivit\'e. D'apr\`es la proposition \ref{cap-nul}, le noyau $\Cap_i(F)$ est nul si et seulement si $\Cap_i(F_n)$ est nul pour tout $n$ ; c'est donc \'equivalent \`a la trivialit\'e de $(X_{\infty}')^0$.
Enfin, $X_{\infty}'$ est $\Z_p$-libre si et seulement s'il n'a pas de $\Z_p$-torsion. Ceci est \'equivalent \`a la trivialit\'e de $(X_{\infty}')^0$ et de $\mu_{X_{\infty}'}$. 
\end{proof}

\textbf{Remarques.}
\begin{itemize}
\item La proposition \ref{desc-cycl} \'etablit un lien entre la structure de $\G_{\infty}'$ et la descente galoisienne dans la classe des extensions cyclotomiques. Dans la suite, le th\'eor\`eme \ref{thm-desc-loc-cycl} examinera le cas o\`u la descente est \'etendue \`a la classe des extensions \textit{localement} cyclotomiques 
\item Dans l'assertion (ii) de la proposition pr\'ec\'edente, on peut remplacer le quantificateur \textit{pour tout entier} par \textit{il existe un entier}.
\end{itemize}

\section{Sur la pro-$p$-libert\'e de $\G_{F_{\infty}}'$}
\subsection{Extensions localement cyclotomiques}
Notre but dans cette partie est d'\'etudier le comportement galoisien des noyaux sauvages \'etales et de mettre en \'evidence le r\^ole particulier jou\'e par les extensions \textit{localement cyclotomiques}. Pour cela, rappelons quelques d\'efinitions et notations.\\
Pour tous les corps $N$ consid\'er\'es, on note $N_{\infty}$ la $\Z_p$-extension cyclotomique de $N$.
\begin{defn}
Une (pro-)$p$-extension $L/F$ de corps de nombres est \textit{localement cyclotomique} si pour toute place finie $v$ de $F$ et toute place $w|v$ de $L$ on a l'inclusion :
$$L_w\subseteq (F_v)_{\infty}.$$
\end{defn}

\textbf{Remarques.}
\begin{enumerate}
\item Dans la terminologie introduite par J.-F. Jaulent (cf. \cite{J3}), les extensions localement cyclotomiques sont exactement les extensions non \textit{logarithmiquement} ramifi\'ees.
\item Une extension localement cyclotomique est $p$-ramifi\'ee.
\end{enumerate}

Consid\'erons $\mathcal L_{F_{\infty}}'=\bigcup_{F\subset L}L ,$ o\`u $L$ parcourt la classe des $p$-extensions localement cyclotomiques de $F$.

L'extension $\mathcal L_{F_{\infty}}'/F$ est une pro-$p$-extension qui contient $F_{\infty}/F$. Dans \cite{JS} les auteurs proposent une construction "explicite" de $\mathcal L_{F_{\infty}}'/F$ au moyen de la tour localement cyclotomique de $F$.

Notons que $\G_{F_{\infty}}'=\Gal(\mathcal L_{F_{\infty}}'/F_{\infty}).$

On pose $\Delta:=\Gal(F(\mu_p)/F)$ (resp. $\Delta_v:=\Gal(F_v(\mu_p)/F_v)$) et $d$ (resp. $d_v$) son ordre.
 
Comme cela est montr\'e dans \cite[Proposition 2.3]{KM} ou dans \cite[Proposition 10]{JM}, et contrairement aux $K$-groupes pairs, le morphisme de norme $$N_{L/F,i} : \WK_{2i}(L)\rightarrow \WK_{2i}(F)$$ n'est pas tout le temps surjectif. C'est exactement dans le cas des extensions localement cyclotomiques (mais non cyclotomiques) que la surjectivit\'e est mise en d\'efaut.

Int\'eressons-nous au morphisme d'extension
$$\WK_{2i}(F)\rightarrow \WK_{2i}(L)^G$$
dans le cas des $p$-extensions localement cyclotomiques.

La suite exacte \ref{se-WK-K} permet de comparer les noyaux de capitulation des noyaux sauvages \'etales et des $K$-groupes pairs dans les extensions localement cyclotomique :

\begin{prop}\label{desc-loc-cycl}
Soit $L/F$ une $p$-extension localement cyclotomique de groupe de Galois $G$. On a l'\'egalit\'e :
$$\Cap_i(L/F)=\ker(\WK_{2i}(F)\rightarrow \WK_{2i}(L)).$$
De plus, on a l'in\'egalit\'e entre les ordres
$$|\WK_{2i}(F)|\geq |\WK_{2i}(L)^G|.$$
En particulier, l'extension $\WK_{2i}(F)\rightarrow \WK_{2i}(L)^G$ est injective si et seulement si elle est bijective.
\end{prop}

Soit $L/F$ une extension localement cyclotomique, disjointe de $F_{\infty}/F$ cyclique de degr\'e $p$. D'apr\`es \cite[Proposition 2.3]{KM}, on sait que $|\coker(N_{L/F,i})|=p$. On a donc la minoration suivante pour les noyaux de capitulation lorsque $i\equiv 0 \mod d$:
$$|\Cap_i(L/F)|\geq \frac{p}{|\ker(N_{L/F,i})|}.$$

On v\'erifie facilement que $|\ker(N_{L/F,i})|= 1 \mbox{ o\`u }p$. Cependant, on ne dispose pas de formules de genres analogues \`a celles donn\'ees dans \cite{KM}, qui nous permettraient de d\'eterminer l'ordre du noyau pr\'ec\'edent dans le cas localement cyclotomique. Comme nous allons le voir dans ce qui suit, le comportement galoisien des noyaux sauvages dans de telles extensions est li\'e \`a la structure du pro-$p$-groupe $\G_{F_{\infty}}'$.

\subsection{Caract\'erisation de la pro-$p$-libert\'e}
On int\'eresse dans cette section au caract\`ere pro-$p$-libre du groupe $\G_{\infty}':=\G_{F_{\infty}}'$. Ce probl\`eme a d\'ej\`a \'et\'e abord\'e par diff\'erents auteurs (cf. \cite{N2}, \cite{O}, \cite{Sha} \cite{Wi}...). 

\textbf{Exemples.}
\begin{itemize}
\item Si $F$ contient $\mu_p$ et est un corps $p$-rationnel (ou $p$-r\'egulier, cf. \cite{GrJ}, \cite{Mo} et \cite{MN}) alors $\G_{\infty}'$ est trivial (donc pro-$p$-libre de rang $0$). Plus g\'en\'eralement $\G_{\infty}'=0$ si et seulement si $\WK_{2i}(F)=0$ pour un entier positif $i$ tel que $i\equiv 0 \mod d$.

\item Si $(X_{\infty}')^0=0$, $\mu(X_{\infty}')=0$ et $\lambda(X_{\infty}')=1$ alors $X_{\infty}'\simeq \Z_p$. Ainsi $\G_{\infty}'\simeq \Z_p$ (donc pro-$p$-libre de rang $1$). Ces conditions sont satisfaites pour les corps suivants (cf. \cite{W}) : $$F=\Q(\mu_p), \mbox{  avec  } p=37,59,67,101,103,131,149,233,257,263.$$
\end{itemize}

On ne connait pas d'exemple de corps de nombres tel que $\G_{\infty}'$ est pro-$p$-libre de rang $d(\G_{\infty}')$ sup\'erieur \`a $1$. Par contre, nous construirons dans la suite des couples $(F,p)$ avec $\G_{\infty}'$ non pro-$p$-libre.

\bigskip

Le th\'eor\`eme suivant caract\'erise la pro-$p$-libert\'e de  $\G_{\infty}'$. On peut trouver une preuve de $(i)\Rightarrow (ii)$ dans \cite[Proposition 4.5]{V}, et une preuve de la r\'eciproque (lorsque $F$ contient $\mu_{2p}$) dans \cite[th\'eor\`eme 3.1]{N3}. Nous proposons ici une d\'emonstration de l'\'equivalence, diff\'erente des deux pr\'ec\'edentes et qui se base essentiellement sur l'\'equivalence montr\'ee dans la proposition \ref{kawada2}.

Nous supposerons que l'invariant "mu" associ\'e \`a $F_{\infty}/F$ est trivial, de sorte que, pour toute $p$-extension $L/F$ localement cyclotomique, l'invariant "mu" associ\'e \`a $L_{\infty}/L$ est aussi trivial (cf. par exemple \cite[Theorem 11.3.8]{NSW}). C'est une cons\'equence de la forme faible de la conjecture de Leopoldt, qui est vraie pour la $\Z_p$-extension cyclotomique.

\medskip
Fixons un entier positif $i$ tel que $i\equiv 0 \mod d$ (on rappelle que $F$ contient $\mu_4$ lorsque $p=2$).
\'Etant donn\'ee une $p$-extension $M/L$ localement cyclotomique contenant $F$, on a le diagramme :
\begin{eqnarray}\label{diag-transfert}
\xymatrix{
\WK_{2i}(L_{\infty}) \ar[r] \ar[d]^{\simeq}                      & \WK_{2i}(M_{\infty}) \ar[d]^{\simeq}       \\
\Q_p/\Z_p(i)\otimes X_{L_{\infty}}'\ar[r]^{\mathcal V}           & \Q_p/\Z_p(i)\otimes X_{M_{\infty}}'
}
\end{eqnarray}
o\`u
\begin{itemize}
\item le morphisme $\mathcal V$ provient du transfert (il est d\'efini dans la proposition \ref{kawada2}),
\item les fl\`eches verticales sont les isomorphismes montr\'es dans la proposition \ref{isom-description},
\item la fl\`eche horizontale sup\'erieure provient par limite inductive des morphismes d'extension $$\WK_{2i}(L_n)\rightarrow \WK_{2i}(M_n).$$
\end{itemize}
\begin{lemme}
Le diagramme \ref{diag-transfert} est commutatif.
\end{lemme}
\begin{proof}
Les fl\`eches horizontales dans le diagramme \ref{diag-transfert} proviennent par dualit\'e de la co-restriction :
$$H^1(\G_{M_{\infty}}',\Q_p/\Z_p)\rightarrow H^1(\G_{L_{\infty}}',\Q_p/\Z_p).$$
\end{proof}
\begin{thm}\label{thm-desc-loc-cycl}
Les assertions suivantes sont \'equivalentes :

(i) Le groupe $\G_{F_{\infty}}'$ est pro-$p$-libre.

(ii) L'invariant $\mu_{X_{F_{\infty}}'}$ est nul et pour toute $p$-extension $L/F$ localement cyclotomique et tout entier $i\equiv 0 \mod d$, on a l'isomorphisme canonique :
$$\WK_{2i}(F)\stackrel{\simeq}{\rightarrow}\WK_{2i}(L)^{G},$$
o\`u $G=\Gal(L/F).$
\end{thm}
\begin{proof}
D'apr\`es la proposition \ref{kawada2}, le groupe $\G_{F_{\infty}}'$ est pro-$p$-libre si et seulement si pour tout sous groupe ouvert $U\triangleleft \G_{E_{\infty}}'$
\begin{enumerate}
\item[(a)] $U^{ab}$ est $\Z_p$-libre.
\item[(b)] Pour tout sous groupe ouvert $V\triangleleft U$, $$\mathcal V :\Q_p/\Z_p\otimes U^{ab}\rightarrow\Q_p/\Z_p\otimes V^{ab},$$
est injectif.
\end{enumerate}
\medskip

Supposons que (a) et (b) sont satisfaites. Soit $L$ une $p$-extension localement cyclotomique de $F$. Le groupe $\G_{L_{\infty}}'$ est un sous-groupe ouvert distingu\'e de $\G_{F_{\infty}}'$. L'hypoth\`ese (b) et la commutativit\'e de (\ref{diag-transfert}) nous donnent
$$\WK_{2i}(F_{\infty})\hookrightarrow \WK_{2i}(L_{\infty}).$$
Soit l'entier $n_0\geq 0$ tel que $F_{n_0}=F_{\infty}\cap L$. Les groupes de Galois $\Gamma_{F_{n_0}}:=\Gal(F_{\infty}/F_{n_0})$ et $\Gamma_{L}:=\Gal(L_{\infty}/L)$ sont \textit{canoniquement} isomorphes. Posons $\Gamma:=\Gamma_{F_{n_0}}\simeq \Gamma_{L}$. On peut alors passer aux invariants sous $\Gamma$ dans l'inclusion pr\'ec\'edente.  L'hypoth\`ese (a) et la proposition \ref{desc-cycl} nous donnent ainsi l'inclusion :
$$\WK_{2i}(F_{n_0})\hookrightarrow \WK_{2i}(L).$$ 
La descente galoisienne des noyaux sauvages est satisfaite dans $F_{\infty}/F$. Ainsi $\Cap_i(L/F)=0$. La surjectivit\'e r\'esulte de la proposition \ref{desc-loc-cycl}.

\medskip
Montrons que l'assertion (ii) entra\^ine (a) et (b).

Soit $U\triangleleft \G_{F_{\infty}}$ un sous-groupe ouvert. Il existe un entier $n_0$ et une $p$-extension $L/F_{n_0}$ localement cyclotomique telle que $U^{ab}=X_{L_{\infty}}'$. Par hypoth\`ese, pour tout entier $m\geq 0$ on a $\Cap_i(L_m/F)=0$ . Donc d'apr\`es la proposition \ref{cap-nul}, on a aussi $\Cap_i(L_m/L)=0$. Enfin la proposition \ref{desc-cycl} nous donne la $\Z_p$-libert\'e de $U^{ab}$ et d\'emontre (a).

Soit un couple de sous-groupes ouverts $V\triangleleft U$. Il existe un couple d'extensions de $F$, localement cyclotomiques, $L\subseteq M$ telles que $$U^{ab}=X_{L_{\infty}}' \mbox{ et } V^{ab}=X_{M_{\infty}}'.$$
Toujours d'apr\`es l'hypoth\`ese (ii) et la proposition \ref{cap-nul}, on a pour tout $n\gg 0$, l'injection
$$\WK_{2i}(L_n)\hookrightarrow \WK_{2i}(M_n).$$

Enfin on a bien $\WK_{2i}(L_{\infty})\hookrightarrow \WK_{2i}(M_{\infty})$ et la commutativit\'e de (\ref{diag-transfert}) permet de conclure \`a l'injectivit\'e de $\mathcal V$.
\end{proof}

\textbf{Remarques.}
\begin{itemize}
\item Encore une fois, dans l'assertion (ii), on peut remplacer le quantificateur \textit{pour tout entier} par le quantificateur \textit{il existe un entier}.
\item La pro-$p$-libert\'e de $\G_{F_{\infty}}'$ est donc \'equivalente \`a la trivialit\'e des deux invariants attach\'es au coupe $(F,p)$ suivants :
$$\mu_{X_{F_{\infty}}'} \mbox{ et } \Cap_i(\L_{F_{\infty}'}/F).$$
\end{itemize}

\section{Applications}
\subsection{Le cas pro-$p$-cyclique}
Un groupe non trivial $\G$ est dit \textit{pro-$p$-cyclique} lorsque $d(\G)=1$. Autrement dit,
$$\G \simeq \Z/p^N \mbox{  ou  } \G \simeq \Z_p,$$

Comme corollaire \`a la d\'emonstration du th\'eor\`eme \ref{thm-desc-loc-cycl}, nous obtenons un crit\`ere pour la pro-$p$-cyclicit\'e de $\G_{F_{\infty}}'$.\\
On fixe un entier positif $i$ tel que $i\equiv 0 \mod d$.
\begin{prop}
Soit $F$ un corps de nombres. S'il existe une extension $L/F$ cyclique de degr\'e $p$, localement cyclotomique et disjointe de $F_{\infty}$ telle que le morphisme
$$\WK_{2i}(F)\rightarrow \WK_{2i}(L)$$
est surjectif, alors le groupe $\G_{F_{\infty}}'$ est pro-$p$-cyclique. 
\end{prop}
\begin{proof}
Posons $M:=\coker(\left(X_{F_{\infty}}'/\Ver(X_{L_{\infty}}')\right))$.\\
On a la suite exacte :
$$X_{F_{\infty}}'(i)\stackrel{\Ver}{\rightarrow}X_{L_{\infty}}'(i)\rightarrow M(i) \rightarrow 0.$$
On pose $\Gamma:=\Gal(F_{\infty}/F)\simeq \Gal(L_{\infty}/L)$. En passant aux co-invariants sous l'action de $\Gamma$, on obtient la suite exacte :
$$\WK_{2i}(F) \rightarrow \WK_{2i}(L) \rightarrow M(i)_{\Gamma} \rightarrow 0.$$
D'apr\`es l'hypoth\`ese, le groupe $M(i)_{\Gamma}$ est nul et le lemme de Nakayama nous donne la trivialift\'e de $M$. Le transfert est ainsi surjectif et de mani\`ere duale, on a l'injection :
$$ H^1(\G_{F_{\infty}}',\Q_p/\Z_p) \stackrel{\cor}{\hookrightarrow} H^1(\G_{L_{\infty}}',\Q_p/\Z_p).$$
L'\'egalit\'e $\cor \circ \res = p$ et l'injectivit\'e de la co-restriction nous donnent donc :
$$\ker(\res)=\!_{p}H^1(\G_{F_{\infty}}',\Q_p/\Z_p) \simeq H^1(\G_{F_{\infty}}',\Z/p).$$
Par ailleurs, la suite exacte d'inflation-restriction nous donne aussi :
$$\ker(\res)=H^1(\Gal(L/F),\Z/p).$$
Or $\Gal(L/F)\simeq \Z/p$, donc $H^1(\G_{F_{\infty}}',\Z/p) \simeq \Z/p$.
\end{proof}

\textbf{Remarque.} R\'eciproquement, on v\'erifie facilement que si $\G_{F_{\infty}}'\simeq \Z_p$, alors pour tout $n\geq 0$,
l'extension induit un isomorphisme
$$\WK_{2i}(F_n)\simeq \WK_{2i}(L_n).$$

\subsection{Crit\`eres de non pro-$p$-libert\'e}
Dans cette partie $F$ d\'esigne un corps de nombres contenant $\mu_{2p}$. On pose $\G_{\infty}':=\G_{F_{\infty}'}$ et pour tout $\Z_p$-module de type fini $A$, on note $\rg_p(A):=\dim_{\F_p}(A/p)$ le $p$-rang de $A$.\\
Dans \cite[Th\'eor\`eme 12]{JS}, (cf aussi \cite[ Th\'eor\`eme 3]{As}) les auteurs donnent un crit\`ere de non-finitude pour $\G_{\infty}$. Il s'\'enonce comme suit : si on a l'in\'egalit\'e
\begin{eqnarray}\label{ineg-fin}
\rg_p(\WK_{2i}(F))\geq 1+2\sqrt{r_2(F)+2},
\end{eqnarray} alors $\G_{\infty}'$ est infini.\\
L'objet de cette partie est d'utiliser le th\'eor\`eme \ref{thm-desc-loc-cycl} pour donner un crit\`ere de non pro-$p$-libert\'e pour $\G_{\infty}'$. Nous aurons besoin du lemme suivant :
\begin{prop}\label{cap} Soit $L$ une $p$-extension $S$-ramifi\'ee de $F$ et $F$ contient $\mu_{2p}$. Posons $G=\Gal(L/F)$ et le $p$-rang de $G$.\\
Si $d(G) > 1+r_2(F)$ alors $\Cap_{i}(L/F)$ n'est pas trivial.
\end{prop}
\begin{proof}
Consid\'erons le diagramme commutatif suivant :
\[
\xymatrix{
0\ar[r] & \K_{2i+1}(L)/p \ar[r] & H^1(G_L^S,\Z/p^n(i+1)) \ar[r] & \K_{2i}(\O_L^S) \ar[r] & 0 \\
0\ar[r] & \K_{2i+1}(F)/p \ar[r] \ar[u] & H^1(G_F^S,\Z/p^n(i+1)) \ar[r] \ar[u]^{\res} & \K_{2i}(\O_F^S) \ar[r] \ar[u] & 0
}
\]
Comme $\mu_{2p} \subset F$ le groupe $G_F^S$ op\`ere trivialement sur $\Z/p(i+1)$ et le noyau de la restriction est :
$$\ker\left(H^1(G_F^S,\Z/p(i+1))\stackrel{res}{\rightarrow} H^1(G_L^S,\Z/p(i+1))\right)=H^1(G,\Z/p)(i+1).$$
Par hypoth\`ese $$d(G)=\dim_{\F_p}(H^1(G,\Z/p)(i+1))>\rg_p(\K_{2i+1}(F))=1+r_2(F).$$
Ainsi, il existe un \'el\'ement non nul, $x\in H^1(G,\Z/p)(i+1)$ tel que $\delta(x)\in\!_{p}\K_{2i}(\O_F^S)$ est \'egalement non nul et qui capitule dans $\K_{2i}(\O_L^S)$.
\end{proof}

\begin{thm}\label{thm-ineg}
Soit $F$ un corps de nombres contenant $\mu_{2p}$, tel que
\begin{eqnarray}\label{ineg-proplibre}\rg_p(\WK_{2i}(F))\geq 1+r_2(F).
\end{eqnarray}
Alors $\G_{\infty}'$ n'est pas un pro-$p$-groupe libre.
\end{thm}
\begin{proof}
Posons $\G'=\Gal(\L_{\infty}'/F)$. Soit $L_0/F$ la sous-extension de $\L_{\infty}'/F$, ab\'eliennne d'exposant $p$ maximale. Comme $\mu_{2p}\subset F$, on a les \'egalit\'es entre $p$-rangs (cf. \cite{J}) : $$d(\Gal(L_0/F))=1+\rg_p(\WK_{2i}(F)).$$
Donc par hypoth`ese $d(\Gal(L_0/F))>1+r_2(F)$. Ainsi, d'apr\`es la proposition \ref{cap}, le noyau $\Cap_{i}(L_0/F)$ est non nul. Comme $L_0/F$ est une $p$-extension localement cyclotomique, la capitulation porte sur le noyau $\WK_{2i}(F)$ et le th\'eor\`eme \ref{thm-desc-loc-cycl} permet de conclure.
\end{proof}
\textbf{Remarques.}
\begin{itemize}
\item L'in\'egalit\'e (\ref{ineg-fin}) entra\^ine l'in\'egalit\'e (\ref{ineg-proplibre}) d\`es que $r_2(E)\geq 5$. Ainsi, un corps satisfaisant aux conditions du th\'eor\`eme \ref{thm-ineg} et de degr\'e au moins $10$ sur $\Q$ admet un groupe $\G_{\infty}'$ infini \textit{et} non pro-$p$-libre.
\item Le $p$-rang du noyau sauvage est accessible par des m\'ethodes num\'eriques puisqu'il co\"incide avec le $p$-rang du groupe des classes logarithmiques (cf. \cite{J}).
\end{itemize}
\begin{corollaire}\label{cor:thm-ineg}
Si au moins une place $p$-adique se ramifie totalement dans $F_{\infty}/F$, alors l'in\'egalit\'e $$\rg_p(A_F')\geq 1+r_2(F)$$ implique que $\G_{\infty}'$ n'est pas pro-$p$-libre.
\end{corollaire} 
\begin{proof}
C'est une cons\'equence de la proposition \ref{prop-surj} et du th\'eor\`eme \ref{thm-ineg}.
\end{proof}

\medskip

Pour v\'erifier que le crit\`ere du th\'eor\`eme \ref{thm-ineg} est int\'eressant il faut exhiber des corps de nombres $F$ qui satisfont l'in\'egalit\'e (\ref{ineg-proplibre}). Le principe est le suivant : dans une $p$-extension $F/k$, le $p$-rang des noyaux sauvages \'etales de $F$ est suffisamment grand d\`es que l'extension $F/k$ est suffisamment ramifi\'ee. 

\begin{prop}
Soit $F/k$ une extension cyclique de degr\'e $p$. Supposons que 
\begin{itemize}
\item $F$ contient $\mu_{2p}$
\item au moins une place $p$-adique se ramifie totalement dans $F_{\infty}/F$
\item $F/k$ est une extension cyclique de degr\'e $p$ dans laquelle toute les places $p$-adiques se ramifient.
\item $|Ram(F/k)-S_p(k)| \geq 2+(p+1)r_2(k)$, o\`u $Ram(F/k)$ est l'ensemble des places de $k$ qui se ramifient dans $F/k$.\\
\end{itemize}
Alors le groupe $\G_{\infty}'$ n'est pas pro-$p$-libre.
\end{prop}
\begin{proof}
On a l'in\'egalit\'e (cf. \cite[Proposition 10.8.3]{NSW}) :
$$\rg_p A_F'\geq |Ram(F/k)-S_p(k)|-r_1(k)-r_2(k)-\delta +r_1'(k),$$
o\`u $r_1'(k)$ est le nombre de places r\'elles de $k$ qui se complexifient dans $F$ et $\delta$ vaut $0$ ou $1$ selon que $\mu_p \not\subset k$ ou $\mu_p \subset k$.\\
La proposition \ref{prop-surj} et le th\'eor\`eme \ref{thm-ineg} permettent de conclure. 
\end{proof}
\textbf{Remarque.} La proposition pr\'ec\'edente permet de construire une infinit\'e de couples $(F,p)$ tels que $\G_{F_{\infty}}'$ n'est pas pro-$p$-libre.

\textbf{Exemple.} Supposons que $F/\Q(\mu_p)$ est une extension cyclique de degr\'e $p$, ramifi\'ee au dessus de $p$ et dans laquelle au moins $\frac{p^2+3}{2}$ places non $p$-adiques se ramifient. Alors le groupe $\G_{F_{\infty}}'$ n'est pas pro-$p$-libre.

Pour finir, on se place dans le cadre suivant. Le premier $p$ est suppos\'e impair. Le corps $E$ contient $\mu_{p}$ et est \`a conjugaison complexe. On note $F$ le sous corps totalement r\'eel de $E$. Posons $\Delta:=\Gal(E/F)$. Pour tout $\Z_p[\Delta]$-module $A$, on note $A^+$ (resp. $A^-$) la composante r\'eelle (resp. imaginaire) de $A$.\\
Pour simplifier, on pose $i=1$ et $\G_{\infty}'=\G_{E_{\infty}}'$.

Pour les corps \`a multiplication complexe, on dispose du principe g\'en\'eral du \textit{Spiegelungsatz} de Leopoldt qui compare les parties "plus" et "moins" (cf. par exemple \cite[Theorem 3.5]{Ko2} ou \cite[Proposition 6]{JM}). Dans le cadre des noyaux sauvages, cela donne l'in\'egalit\'e suivante :
$$0\leq \dim_{\F_p}(\WK_{2}(E)/p)^+-\dim_{\F_p}(\WK_{2}(E)/p)^-\leq [F:\Q]=r_2(E).$$

On remarque que si le noyau sauvage satisfait l'in\'egalit\'e \ref{ineg-proplibre} alors n\'ecessairement la partie imaginaire $\WK_{2}(E)^-$ n'est pas triviale.\\
En fait, comme le montre la proposition suivante, cette condition est suffisante pour montrer que $\G_{\infty}'$ n'est pas pro-$p$-libre.
\begin{prop}
Si $\WK_{2}(E)^-\neq 0$ alors le groupe $\G_{\infty}'$ n'est pas pro-$p$-libre.
\end{prop}
\begin{proof}
L'hypoth\`ese $\WK_{2}(E)^-$ non trivial montre que $F$ admet une extension localement cyclotomique, disjointe de $F_{\infty}$. On a donc $$\dim_{\F_p}(\Gal(L_0/E)^+)\geq 2.$$
Il suffit ensuite de prendre la partie $+$ de la suite exacte :
$$0\rightarrow \K_{3}(E)/p \rightarrow H^1(G_E^S,\mu_p)\otimes \mu_p \stackrel{\delta}{\rightarrow}\!_{p}\K_{2}(\O_E^S)\rightarrow 0,$$
en consid\'erant le fait que $(\K_{3}(E)/p)^+$ est cyclique.\\
Le raisonnement est alors le m\^eme que dans la proposition \ref{cap} : il existe un \'el\'ement de $H^1(G_E^S,\mu_p)\otimes \mu_p$ qui n'est pas dans $\K_{3}(E)/p$ et qui se trivialise par restriction dans $H^1(G_{L_0}^S,\mu_p)\otimes \mu_p$. Cet \'el\'ement s'envoie donc sur un \'el\'ement non nul de $\Cap_i(L_0/F)$. Enfin le th\'eor\`eme \ref{thm-desc-loc-cycl} permet de conclure.
\end{proof}

Nous retrouvons enfin la proposition 3.3 de \cite{Wi} : 
\begin{corollaire}\label{cor-partieplus}
Supposons qu'au moins une place $p$-adique se ramifie totalement dans $E_{\infty}/E$, et que $F$ admet une extension non-ramifi\'ee et $p$-d\'ecompos\'ee (i.e. $A_F'\neq 0$). Alors le groupe $\G_{E_{\infty}}'$ n'est pas pro-$p$-libre.
\end{corollaire}

La proposition pr\'ec\'edente est imm\'ediate si l'on suppose que $F$ satifait la conjecture de Greenberg en $p$. Cette conjecture postule en effet la trivialit\'e de l'invariant $\lambda_{X_{\infty}'}^+$ (i.e. $\left(X_{\infty}'\right)^+$ est un $\Z_p$-module \textit{fini}).\\
En passant \`a la partie "moins" dans l'isomorphisme $$\WK_{2}(E)/p\simeq (X_{\infty}'\otimes \mu_p)_{\Gamma},$$
on remarque que $\WK_{2}(E)^-\neq 0$ entra\^ine la non nullit\'e $(X_{\infty}')^+$. La conjecture de Greenberg entraine que $(\G_{E_{\infty}}')^{ab}$ contient un sous-groupe fini \textit{non nul}. Ainsi $\G_{E_{\infty}}'$ n'est pas pro-$p$-libre.

\end{document}